\documentclass{amsart}
\usepackage{amsfonts}
\usepackage{amsmath}
\usepackage{amssymb}
\usepackage{amsthm}
\usepackage{graphicx}
\usepackage{capt-of}
\usepackage{tikz}
\usetikzlibrary{patterns}
\usetikzlibrary{decorations.markings}
\usepackage{multirow}
\usepackage{array} 
\newcolumntype{C}{>{$}c<{$}}

\newtheorem{theorem}{Theorem}

\newtheorem{lemma}{Lemma}
\newtheorem{rem}{Remark}
\newtheorem{exmp}{Example}

\begin{document}

\title{Connected sums of $Z$-knotted triangulations}

\begin{abstract}
An embedded graph is called $z$-knotted if it contains a unique zigzag (up to reversing).
We consider $z$-knotted triangulations, i.e. $z$-knotted embedded graphs whose faces are triangles,
and describe all cases when the connected sum of two $z$-knotted triangulations is $z$-knotted.
\end{abstract}
\author{Mark Pankov, Adam Tyc}
\subjclass[2000]{}
\address{Faculty of Mathematics and Computer Science, 
University of Warmia and Mazury, S{\l}oneczna 54, Olsztyn, Poland}
\email{pankov@matman.uwm.edu.pl, adam121994@wp.pl}

\maketitle

\hfill{\it To the memory of Michel Marie Deza}

\section{Introduction}
The notion of {\it Petrie polygon} for polytopes was introduced in Coxeter's book \cite{Cox73}
(see also \cite{DDS-paper,DP, Williams} and \cite[Chapter 8]{D-DS} for some generalizations).
For graphs embedded in $2$-dimensional surfaces the same objects appear as 
{\it zigzags} in \cite{zig1, zig2, D-DS},  {\em geodesics} in  \cite{GrunMo63} and  {\it left-right paths} in \cite{Sh75}.
Following \cite{zig1,zig2,DDS-paper,D-DS} we call them {\it zigzags}.

A zigzag is a cyclic sequence of vertices in an embedded graph $\Gamma$,
where any two consecutive vertices form an edge, any three consecutive vertices are mutually distinct vertices on a face
and any four consecutive vertices are not on the same face.
For every zigzag the reversed sequence also is a zigzag and we do not distinguish such two zigzags.
When there is a unique zigzag (up to reversing), we say that $\Gamma$ is $z$-{\it knotted}.

The notion of zigzag is closely connected to the notion of {\it central circuit} in a $4$-regular embedded graph.
Zigzags and central circuits are interesting for many reasons. 
For example, there is an application to knot theory related to projections of links \cite[Section 1.3]{D-DS}.
In particular, $z$-knotted embedded graphs connect to the {\it Gauss code problem} as follows.
The medial graph of such a graph has a unique central circuit,
and this medial graph can be considered as a closed  curve having simple self-intersections only,
i.e. it is a realization of a certain Gauss code on the corresponding surface
(see \cite{LinsOL-Silva}).

The main objects of this paper are $z$-knotted {\it triangulations} ($z$-knotted embed\-ded graphs whose faces are triangles). 
Note that the dual objects, i.e. $z$-knotted $3$-re\-gular graphs, are considered in \cite[Section 2.3]{D-DS}.
A simple example of $z$-knotted triangulation is the graph of the $n$-gonal  bipyramid when $n$ is odd 
(this graph is not $z$-knotted if $n$ is even).
Examples of $z$-knotted fullerenes can be found in \cite[Section 2.3]{D-DS} (the dual graphs are triangulations).
In this paper, we show how construct more examples. 

The connected sum of manifolds  is a well-known topological notion.
In a similar way, we define the connected sum of triangulations $\Gamma$ and $\Gamma'$
as the gluing by two distinguished faces $F$ and $F'$ (in $\Gamma$ and $\Gamma'$, respectively).
This is a triangulation in the connected sum of the surfaces containing $\Gamma$ and $\Gamma'$,
and it depends on an identification homeomorphism between $\partial F$ and $\partial F'$.
We describe all cases when the connected sum of two $z$-knotted triangulations is $z$-knotted.

\section{Zigzags in embedded graphs}
Let $M$ be a connected closed $2$-dimensional surface (not necessarily orientable).
Let also $\Gamma$ be a finite simple connected graph embedded in $M$ such that every vertex of $\Gamma$ is a point on $M$,
every edge is homeomorphic to the segment $[0;1]$ and any two edges are disjoint or intersecting in a  vertex.
The {\it faces} of $\Gamma$ are defined as the closures of the connected components of $M\setminus \Gamma$.
We suppose that our graph embedding is a {\it  closed $2$-cell embedding},  i.e. every face is homeomorphic to  a closed $2$-dimensional disc.
Also, we require that the following two conditions hold:
\begin{enumerate}
\item[$\bullet$] every edge is contained in precisely  two distinct faces,
\item[$\bullet$] the intersection of two distinct faces is an edge or a vertex or empty.
\end{enumerate}
A {\it zigzag}  in $\Gamma$ is a sequence of vertices $\{x_{i}\}_{i\in {\mathbb N}}$ satisfying the following three conditions for every $i\in {\mathbb N}$:
\begin{enumerate}
\item[(Z1)] $x_{i},x_{i+1}$ are distinct vertices on a certain edge,
\item[(Z2)] there is a unique face containing $x_{i},x_{i+1},x_{i+2}$,
\item[(Z3)] the faces containing $x_{i},x_{i+1},x_{i+2}$ and $x_{i+1},x_{i+2},x_{i+3}$ are distinct.
\end{enumerate} 
Let $Z=\{x_{i}\}_{i\in {\mathbb N}}$ be a zigzag in $\Gamma$.
Since our graph is finite, there is a natural number $n>3$ such that $x_{i+n}=x_{n}$ for every $i\in {\mathbb N}$. 
The smallest number $n$ satisfying this condition is called the {\it length} of $Z$.
So, our zigzag is the cyclic sequence $x_{1},\dots, x_{n}$, where $n$ is the length of $Z$.
For every natural $j$ the sequence $\{x_{i+j}\}_{i\in {\mathbb N}}$ is a zigzag,
and all such zigzags are identified with $Z$.

Suppose that $x_{1},x_{2},x_{3}$ are three distinct vertices on a certain face and $x_{2}$ is adjacent to both $x_{1},x_{3}$. 
Consider the second face containing $x_{2},x_{3}$.
It contains the unique vertex $x_4$ adjacent to $x_{3}$ and distinct from $x_{2}$.
We apply the same arguments to $x_{2},x_{3},x_{4}$ and get a vertex $x_{5}$.
Step by step, we construct a sequence $\{x_{i}\}_{i\in {\mathbb N}}$ satisfying the conditions (Z1) -- (Z3).
Therefore, every zigzag $\{x_{i}\}_{i\in {\mathbb N}}$ is completely determined by any three consecutive vertices $x_{i},x_{i+1},x_{i+2}$.

If $X=\{x_{1},\dots, x_{n}\}$ is a sequence of vertices in $\Gamma$, 
then $X^{-1}$ denotes the reversed sequence $x_{n},\dots,x_{1}$.
If $Z$ is a zigzag, then the same holds for $Z^{-1}$.
If $Z$ contains a sequence $x,y, z$, then the sequence $z,y,x$ is contained the reversed zigzag $Z^{-1}$.
A sequence of type $x,y,z,\dots,z,y,x,\dots$ cannot be a zigzag 
(see, for example, \cite{DP}), i.e. a zigzag is not reversed to itself.
Every zigzag will be identified with its reverse.

\begin{exmp}{\rm
See Fig.1 for examples of zigzags in  the complete graph $K_{4}$ and the cube graph $Q_3$.
\begin{center}
\begin{tikzpicture}[scale=0.8]

\draw[fill=black] (0,0.5) circle (3pt);
\draw[fill=black] (3,0.5) circle (3pt);
\draw[fill=black] (1.5,1.5) circle (3pt);
\draw[fill=black] (1.5,3.5) circle (3pt);

\draw[thick, dashed] (0,0.5) -- (1.5,1.5);
\draw[thick, dashed] (3,0.5) -- (1.5,3.5);
\draw[thick,line width=1.25pt] (0,0.5) -- (3,0.5) -- (1.5,1.5) -- (1.5,3.5) -- (0,0.5);

\draw[fill=black] (5.5,0.5) circle (3pt);
\draw[fill=black] (5.5,3.5) circle (3pt);
\draw[fill=black] (8.5,0.5) circle (3pt);
\draw[fill=black] (8.5,3.5) circle (3pt);

\draw[fill=black] (6.25,1.25) circle (3pt);
\draw[fill=black] (6.25,2.75) circle (3pt);
\draw[fill=black] (7.75,1.25) circle (3pt);
\draw[fill=black] (7.75,2.75) circle (3pt);

\draw[thick, dashed] (6.25,1.25) -- (6.25,2.75) -- (7.75,2.75);
\draw[thick, dashed] (6.25,2.75) -- (5.5,3.5);
\draw[thick, dashed] (7.75,1.25) -- (8.5,0.5);
\draw[thick, dashed] (5.5,0.5) -- (8.5,0.5) -- (8.5,3.5);
\draw[thick,line width=1.25pt] (5.5,0.5) -- (6.25,1.25) -- (7.75,1.25) -- (7.75,2.75) -- (8.5,3.5) -- (5.5,3.5) -- (5.5,0.5);

\end{tikzpicture}
\captionof{figure}{ }
\end{center}
}\end{exmp}

For a zigzag $\{x_{i}\}_{i\in {\mathbb N}}$ in $\Gamma$ the sequence $\{F_{i}\}_{i\in {\mathbb N}}$, where $F_i$ is the face containing 
$x_{i},x_{i+1},x_{i+2}$, is a zigzag in the dual graph $\Gamma^{*}$.
There is a one-to-one correspondence between zigzags in $\Gamma$ and $\Gamma^{*}$.

\section{$z$-knotted graphs}
The graph $\Gamma$ is called $z$-{\it knotted} if it contains the unique zigzag (up to reversing).

Let $e$ be an edge in $\Gamma$ and let $x,y$ be the vertices on this edge.
We take one of the two faces containing $e$. 
This face contains a unique vertex $p\ne y$ adjacent to $x$ and the unique vertex $q\ne x$ adjacent to $y$
(if $p=q$, then the face is a triangle).
The sequences $p,x,y$ and $x,y,q$ define some zigzags (not necessarily distinct)
and the reversed zigzags correspond to the sequences $y,x,p$ and $q,y,x$, respectively. 
See Fig.2.
If $\Gamma$ is $z$-knotted, then the zigzags determined by $p, x, y$ and $x, y, q$ must coincide (up to reversal), 
and so the unique zigzag contains, say, $p, x, y$ and either $x, y, q$ or $q, y, x$, i.e.
one of the following possibilities is realized:
\begin{enumerate}
\item[(1)] the unique zigzag is a sequence $\dots,x,y,\dots, y,x,\dots$,
\item[(2)] the unique zigzag is a sequence $\dots,x,y\dots,x,y,\dots$.
\end{enumerate}
The edge $e$ is said to be of {\it first} or of {\it second type} if the corresponding case is realized.
So, the zigzag in a $z$-knotted graph passes through every edge twice in different directions (first type) or in the same direction (second type).
Conversely, if there is a zigzag of $\Gamma$ passing through every edge twice, then $\Gamma$ is $z$-knotted.

\begin{exmp}{\rm
Consider the graph of the $n$-gonal bipyramid $BP_{n}$, $n\ge 3$ 
which consists of an $n$-gone whose vertices are denoted by $1,\dots,n$ and connected with two disjoint vertices $a,b$.
Suppose that $n=2k$. 
If $k$ is odd, then the sequence
$$a, 1, 2, b,3,4,\dots, a, n-1, n, $$$$b,1, 2, a, 3, 4,\dots, b, n-1, n$$
is a zigzag. 
In the case when $k$ is even, one of the zigzags is $a,1,2,b,3,4$ if $k=2$ and 
$$a, 1, 2, b, 3, 4, a,\dots, b, n-1, n$$
if $k\ge 4$.
Each of these sequences passes through the edge joining $a$ and $1$ only ones.
This means that $BP_{n}$ is not $z$-knotted if $n$ is even.
Consider the case when $n$ is odd, i.e. $n=2k+1$ and $k\ge 1$.
If $k=1$, then
$$a,1,2,b,3,1,a,2,3,b,1,2,a,3,1,b,2,3$$
is a zigzag.
Similarly, if $k$ is odd and not less than $3$, then 
$$a, 1, 2, b, 3,4,\dots, a, n-2, n-1, b, n, 1, 
a, 2, 3, b, 4,5,\dots,a, n-1, n,$$
$$b, 1, 2, a, 3,4,\dots, b, n-2, n-1, a, n, 1,
b, 2, 3, a, 4,5,\dots,b, n-1, n$$
is a zigzag.
In the case when $k=2$, the sequence 
$$a,1,2,b,3,4,a,5,1,b,2,3,a,4,5,$$
$$b,1,2,a,3,4,b,5,1,a,2,3,b,4,5$$
is a zigzag.
If $k$ is even and not less than $4$, then 
$$a, 1, 2, b, 3, 4, a, \dots, b, n-2, n-1, 
a, n, 1, b, 2, 3, a, \dots, a, n-1, n, $$
$$b, 1, 2, a, 3, 4, b, \dots, a, n-2, n-1,
b, n, 1,a, 2, 3, b, \dots, b, n-1, n$$
is a zigzag.
The length of these zigzags is $6n$. Since $BP_{n}$ contains precisely $3n$ edges, 
each of these zigzags passes through all edges twice.
This guarantees that there are no other zigzags (except the reverse) and $BP_{n}$ is $z$-knotted if $n$ is odd.
Every edge joining $x\in\{a,b\}$ and $y\in\{1,\dots,n\}$ is of first type and all edges in the $n$-gon are of second type.
Each face of $BP_{n}$ is a triangle formed by two edges of first type and one edge of second type.
The zigzag passes through each face thrice.
Suppose that $x\in\{a,b\}$ and $y,z\in \{1,\dots,n\}$ are the vertices on a certain face.
An immediate verification shows that the zigzag is a sequence 
$$\dots,x, y, z, \dots, y, x, z, \dots, y, z, x,\dots$$
if $k$ is odd and it is a sequence 
$$\dots,x, y, z, \dots, y, z, x, \dots, y, x, z,\dots$$
if $k$ is even.
}\end{exmp}

Suppose that $\Gamma$ is $z$-knotted. This is equivalent to the dual graph $\Gamma^{*}$ being $z$-knotted.
If $e$ is an edge of first type in $\Gamma$, then the zigzag is a sequence 
$$\dots,p,x,y,p'\dots, q,y,x,q'\dots,$$
see Fig.2.
Denote by $F$ and $F'$ the faces containing $p,x,y,q$ and $p',x,y,q'$, respec\-tively. 
The edge $e$ corresponds to the edge of $\Gamma^{*}$ joining  $F$ and $F'$.
The zigzag in $\Gamma^{*}$ is a sequence $\dots,F,F',\dots,F,F',\dots$, i.e. 
the edge of $\Gamma^{*}$ corresponding to $e$ is of second type.
\begin{center}
\begin{tikzpicture}[scale=0.8]

\draw[fill=black] (1,0) circle (3pt);
\draw[fill=black] (3,0) circle (3pt);
\draw[fill=black] (0.25,1.5) circle (3pt);
\draw[fill=black] (3.75,1.5) circle (3pt);
\draw[fill=black] (0.25,-1.5) circle (3pt);
\draw[fill=black] (3.75,-1.5) circle (3pt);

\draw[thick] (0.25,1.5) -- (1,0)-- (3,0)-- (3.75,1.5);
\draw[thick] (0.25,-1.5) -- (1,0)-- (3,0)-- (3.75,-1.5);

\node at (-0.05,1.65) {$p$};
\node at (4.05, 1.65) {$q$};
\node at (-0.05,-1.65) {$q'$};
\node at (4.05, -1.65) {$p'$};
\node at (1.15,-0.25) {$x$};
\node at (2.85,-0.25) {$y$};
\node at (2,0.15) {$e$};
\node at (2,1) {$F$};
\node at (2,-1) {$F'$};

\end{tikzpicture}
\captionof{figure}{ }
\end{center}
Similarly, every edge of second type in $\Gamma$ corresponds to an edge of first type in $\Gamma^{*}$.

\section{Two types of faces in $z$-knotted triangulations}
In this section, we suppose that $\Gamma$ is $z$-knotted and every face of $\Gamma$ is a triangle, 
i.e. $\Gamma$ is a $z$-knotted triangulation.
\begin{lemma}\label{lemma1}
For every face $F$ one of the following possibilities is realized:
\begin{enumerate}
\item[(1)] only one edge in $F$ is of second type,
\item[(2)] all edges in $F$ are of second type.
\end{enumerate}
\end{lemma}

In the cases (1) and (2), we say that $F$ is a $(1,1,2)$-face or a $(2,2,2)$-face, respectively.

\begin{proof}[Proof of Lemma \ref{lemma1}]
Observe that the zigzag passes through the face $F$ thrice;
i.e. there are precisely three intersections of the zigzag with $F$ which contain more than one vertex.
Every such intersection consists of two edges. 
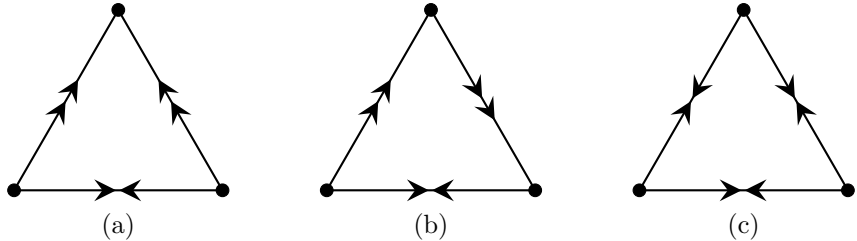
\begin{center}
\begin{tikzpicture}[scale=0.8]

\draw[fill=black] (0,2) circle (3pt);
\draw[fill=black] (-1.7320508076,-1) circle (3pt);
\draw[fill=black] (1.7320508076,-1) circle (3pt);

\draw [thick, decoration={markings,
mark=at position 0.49 with {\arrow[scale=2,>=stealth]{<}},
mark=at position 0.61 with {\arrow[scale=2,>=stealth]{<}}},
postaction={decorate}] (0,2) -- (-1.7320508076,-1);

\draw [thick, decoration={markings,
mark=at position 0.49 with {\arrow[scale=2,>=stealth]{>}},
mark=at position 0.61 with {\arrow[scale=2,>=stealth]{<}}},
postaction={decorate}] (-1.7320508076,-1) -- (1.7320508076,-1);

\draw [thick, decoration={markings,
mark=at position 0.49 with {\arrow[scale=2,>=stealth]{<}},
mark=at position 0.61 with {\arrow[scale=2,>=stealth]{<}}},
postaction={decorate}] (0,2) -- (1.7320508076,-1);

\node at (0,-1.6) {(a)};

\draw[fill=black] (5.1961524228,2) circle (3pt);
\draw[fill=black] (3.4641016152,-1) circle (3pt);
\draw[fill=black] (6.9282032304,-1) circle (3pt);

\draw [thick, decoration={markings,
mark=at position 0.49 with {\arrow[scale=2,>=stealth]{<}},
mark=at position 0.61 with {\arrow[scale=2,>=stealth]{<}}},
postaction={decorate}] (5.1961524228,2) -- (3.4641016152,-1);

\draw [thick, decoration={markings,
mark=at position 0.49 with {\arrow[scale=2,>=stealth]{>}},
mark=at position 0.61 with {\arrow[scale=2,>=stealth]{<}}},
postaction={decorate}] (3.4641016152,-1) -- (6.9282032304,-1);

\draw [thick, decoration={markings,
mark=at position 0.49 with {\arrow[scale=2,>=stealth]{>}},
mark=at position 0.61 with {\arrow[scale=2,>=stealth]{>}}},
postaction={decorate}] (5.1961524228,2) -- (6.9282032304,-1);

\node at (5.1961524228,-1.6) {(b)};

\draw[fill=black] (10.3923048456,2) circle (3pt);
\draw[fill=black] (8.660254038,-1) circle (3pt);
\draw[fill=black] (12.1243556532,-1) circle (3pt);

\draw [thick, decoration={markings,
mark=at position 0.49 with {\arrow[scale=2,>=stealth]{>}},
mark=at position 0.61 with {\arrow[scale=2,>=stealth]{<}}},
postaction={decorate}] (10.3923048456,2) -- (8.660254038,-1);

\draw [thick, decoration={markings,
mark=at position 0.49 with {\arrow[scale=2,>=stealth]{>}},
mark=at position 0.61 with {\arrow[scale=2,>=stealth]{<}}},
postaction={decorate}] (8.660254038,-1) -- (12.1243556532,-1);

\draw [thick, decoration={markings,
mark=at position 0.49 with {\arrow[scale=2,>=stealth]{>}},
mark=at position 0.61 with {\arrow[scale=2,>=stealth]{<}}},
postaction={decorate}] (10.3923048456,2) -- (12.1243556532,-1);

\node at (10.3923048456,-1.6) {(c)};

\end{tikzpicture}
\captionof{figure}{ }
\end{center}
This observation shows that the following possibilities cannot be realized:
$F$ contains precisely two edges of second type (see Fig.3 (a) and (b)) 
and all edges in $F$ are of first type (see Fig.3 (c)).
\end{proof}

Suppose that $F$ is a $(1,1,2)$-face. Let $x,y,z$ be the vertices of $F$. 
We assume that the vertices $y,z$ are on a unique edge of second type 
and the zigzag goes twice from $y$ to $z$, see Fig.4.

\begin{center}
\begin{tikzpicture}[scale=0.8]

\draw[fill=black] (0,2) circle (3pt);
\draw[fill=black] (-1.7320508076,-1) circle (3pt);
\draw[fill=black] (1.7320508076,-1) circle (3pt);

\draw [thick, decoration={markings,
mark=at position 0.49 with {\arrow[scale=2,>=stealth]{>}},
mark=at position 0.61 with {\arrow[scale=2,>=stealth]{<}}},
postaction={decorate}] (0,2) -- (-1.7320508076,-1);

\draw [thick, decoration={markings,
mark=at position 0.49 with {\arrow[scale=2,>=stealth]{>}},
mark=at position 0.61 with {\arrow[scale=2,>=stealth]{>}}},
postaction={decorate}] (-1.7320508076,-1) -- (1.7320508076,-1);

\draw [thick, decoration={markings,
mark=at position 0.49 with {\arrow[scale=2,>=stealth]{>}},
mark=at position 0.61 with {\arrow[scale=2,>=stealth]{<}}},
postaction={decorate}] (0,2) -- (1.7320508076,-1);

\node at (-2,-1) {$y$};
\node at (2,-1) {$z$};
\node at (0.3,2) {$x$};

\end{tikzpicture}
\captionof{figure}{ }
\end{center}
Then the zigzag contains the sequences $x,y,z$ and $y,z,x$.
The zigzag passes through $F$ thrice and the third sequence formed by the vertices of $F$ is $y,x,z$.
Since the zigzag is a cyclic sequence, we can suppose $x,y,z$ are the first three consecutive vertices in the zigzag.
The remaining two sequences associated to $F$ are $y,z,x$ and $y,x,z$.
There are precisely the following two possibilities for the zigzag:
\begin{enumerate}
\item[(od)] $x,y,z,\dots,y,x,z,\dots,y,z,x,\dots$,
\item[(ev)] $x,y,z,\dots,y,z,x,\dots,y,x,z,\dots$.
\end{enumerate}
The first case is realized for all faces in the bipyramid graph $BP_{2k+1}$ if $k$ is odd. In this case, we say that $F$ is a $(1,1,2)$-face of {\it odd type}.
The second case corresponds to all faces in $BP_{2k+1}$ if $k$ is even,
and $F$ is said to be a $(1,1,2)$-face of {\it even type} in this case.

Let $F$ be a $(2,2,2)$-face and let $x,y,z$ be the vertices on $F$.
Formally, we have  two possibilities (Fig.5); but the second cannot happen because the zigzag passes through each face thrice.
So, the zigzag defines the orientation on $F$.
We suppose that this orientation coincides with the order of vertices in the sequence $x,y,z$ (the first triangle on Fig.5).
\begin{center}
\begin{tikzpicture}[scale=0.8]

\draw[fill=black] (0,2) circle (3pt);
\draw[fill=black] (-1.7320508076,-1) circle (3pt);
\draw[fill=black] (1.7320508076,-1) circle (3pt);

\draw [thick, decoration={markings,
mark=at position 0.49 with {\arrow[scale=2,>=stealth]{>}},
mark=at position 0.61 with {\arrow[scale=2,>=stealth]{>}}},
postaction={decorate}] (0,2) -- (-1.7320508076,-1);

\draw [thick, decoration={markings,
mark=at position 0.49 with {\arrow[scale=2,>=stealth]{>}},
mark=at position 0.61 with {\arrow[scale=2,>=stealth]{>}}},
postaction={decorate}] (-1.7320508076,-1) -- (1.7320508076,-1);

\draw [thick, decoration={markings,
mark=at position 0.49 with {\arrow[scale=2,>=stealth]{<}},
mark=at position 0.61 with {\arrow[scale=2,>=stealth]{<}}},
postaction={decorate}] (0,2) -- (1.7320508076,-1);

\node at (-2,-1) {$y$};
\node at (2,-1) {$z$};
\node at (0.3,2) {$x$};

\draw[fill=black] (5,2) circle (3pt);
\draw[fill=black] (3.2679491924,-1) circle (3pt);
\draw[fill=black] (6.7320508076,-1) circle (3pt);

\draw [thick, decoration={markings,
mark=at position 0.49 with {\arrow[scale=2,>=stealth]{>}},
mark=at position 0.61 with {\arrow[scale=2,>=stealth]{>}}},
postaction={decorate}] (5,2) -- (3.2679491924,-1);

\draw [thick, decoration={markings,
mark=at position 0.49 with {\arrow[scale=2,>=stealth]{>}},
mark=at position 0.61 with {\arrow[scale=2,>=stealth]{>}}},
postaction={decorate}] (3.2679491924,-1) -- (6.7320508076,-1);

\draw [thick, decoration={markings,
mark=at position 0.49 with {\arrow[scale=2,>=stealth]{>}},
mark=at position 0.61 with {\arrow[scale=2,>=stealth]{>}}},
postaction={decorate}] (5,2) -- (6.7320508076,-1);

\end{tikzpicture}
\captionof{figure}{ }
\end{center}
The zigzag contains the sequences $x,y,z$ and $y,z,x$ and $z,x,y$.
Since the zigzag is a cyclic sequence, we can suppose that the first three consecutive vertices are $x,y,z$. 
Then there are only the following two possibilities for the zigzag:
\begin{enumerate}
\item[(1)] $x, y, z, \dots, z, x, y, \dots, y, z, x,\dots$,
\item[(2)] $x, y, z, \dots, y, z, x, \dots, z, x, y,\dots$.
\end{enumerate}
We say that $F$ is a $(2,2,2)$-face of {\it first} or of {\it second type} if the corresponding case is realized.
See Examples 4 and 5 for $z$-knotted triangulations containing $(2,2,2)$-faces of first type.
In Section 7, we construct a $z$-knotted triangulation with a $(2,2,2)$-face of second type.

\section{Main result}
Let $\Gamma$ and $\Gamma'$ be triangulations of closed connected $2$-dimensional  surfaces $M$ and $M'$, respectively.
Suppose that $F$ is a face in $\Gamma$ and $F'$ is a face in $\Gamma'$.
These faces both are homeomorphic to a closed $2$-dimensional disc 
and each of the boundaries $\partial F$ and $\partial F'$ is the sum of three edges. 
Let $g: \partial F \to \partial F'$ be a homeomorphism transferring every vertex of $F$ to a vertex of $F'$, i.e.
if $x_{i}$, $i\in \{1,2,3\}$ are the vertices of $F$, then $x'_{i}=g(x_{i})$, $i\in \{1,2,3\}$ are the vertices of $F'$.
The {\it connected sum} $\Gamma \#_{g} \Gamma'$ is the graph whose vertex set is the union of the vertex sets of $\Gamma$ and $\Gamma'$,
where every $x_i$ is identified with $x'_i$, and the edge set is the union of the edge sets of $\Gamma$ and $\Gamma'$, 
where the edge containing $x_i,x_j$ is identified with the edge containing $x'_i,x'_j$.
This is a triangulation in the connected sum of $M$ and $M'$.

\begin{exmp}{\rm
Any connected sum of two exemplars of $K_{4}$ is the bipyramid graph $BP_{3}$.
One of the connected sums of two exemplars of $BP_{3}$ is presented on Fig.6.
\begin{center}
\begin{tikzpicture}[scale=0.8]

\draw[fill=black] (0,0) circle (3pt);
\draw[fill=black] (2,0) circle (3pt);
\draw[fill=black] (1,1) circle (3pt);
\draw[fill=black] (3,1) circle (3pt);
\draw[fill=black] (1.5,-1.5) circle (3pt);
\draw[fill=black] (0.75,2.5) circle (3pt);
\draw[fill=black] (2.25,2.5) circle (3pt);

\draw[thick] (0,0) -- (2,0)-- (3,1)-- (1,1)-- (0,0)-- (0.75,2.5);
\draw[thick] (2,0) -- (1,1) -- (0.75,2.5) -- (2,0) -- (2.25,2.5)-- (3,1);
\draw[thick] (2.25,2.5)-- (1,1);
\draw[thick] (0,0)-- (1.5,-1.5) -- (2,0);
\draw[thick] (1,1)-- (1.5,-1.5) -- (3,1);

\end{tikzpicture}
\captionof{figure}{ }
\end{center}
Later, we show that it is a $z$-knotted triangulation containing $(2,2,2)$-faces of first type.
}\end{exmp}

\begin{theorem}
Suppose that $\Gamma$ and $\Gamma'$ are $z$-knotted.
Then the following assertions are fulfilled: 
\begin{enumerate}
\item[{\rm(1)}] 
If $F$ and $F'$ are $(1,1,2)$-faces in $\Gamma$ and $\Gamma'$ {\rm(}respectively{\rm)}, 
then there is a homeomorphism $g: \partial F \to \partial F'$ transferring vertices to vertices and such that 
the connected sum $\Gamma \#_{g} \Gamma'$ is $z$-knotted.
\item[{\rm (2)}] 
If $F$ is a $(2,2,2)$-face of first type in $\Gamma$, 
then for every face $F'$ in $\Gamma'$ and every homeomorphism $g: \partial F \to \partial F'$ transferring vertices to vertices 
the connected sum $\Gamma \#_{g} \Gamma'$ is $z$-knotted.
\item[{\rm (3)}]
Suppose that $F$ is a $(2,2,2)$-face of second type in $\Gamma$ and $F'$ is a face in $\Gamma'$ such that
the connected sum $\Gamma \#_{g} \Gamma'$ is $z$-knotted for a certain homeomorphism $g: \partial F \to \partial F'$ transferring vertices to vertices.
Then $F'$ is a $(2,2,2)$-face of first type or a $(1,1,2)$-face of odd type. 
In these cases, $\Gamma \#_{g} \Gamma'$ is $z$-knotted for every homeomorphism $g: \partial F \to \partial F'$ transferring vertices to vertices.
\end{enumerate}
\end{theorem}

Using this result we construct the following examples of $z$-knotted triangulations containing $(2,2,2)$-faces of first type:
\begin{enumerate}
\item[$\bullet$] If $F$ is a $(1,1,2)$-face of odd type in a $z$-knotted triangulation $\Gamma$ 
and $F'$ is a face in the bipyramid graph $BP_{2k+1}$, where $k$ is odd, 
then for some homeomorphisms $g: \partial F \to \partial F'$ transferring vertices to vertices 
the connected sum $\Gamma \#_{g} BP_{2k+1}$ is $z$-knotted and contains $(2,2,2)$-faces of first type
(Example 4).
\item[$\bullet$]
If $F$ is a $(1,1,2)$-face of even type in a $z$-knotted triangulation $\Gamma$ 
and $F'$ is a face in the bipyramid graph $BP_{2k+1}$, where $k$ is even, 
then for some homeomorphisms $g: \partial F \to \partial F'$ transferring vertices to vertices 
the connected sum $\Gamma \#_{g} BP_{2k+1}$ is $z$-knotted and contains $(2,2,2)$-faces of first type
(Example 5). 
\end{enumerate}

\begin{rem}{\rm
Examples of $z$-knotted fullerenes can be found in \cite[p.31, Fig 2.2]{D-DS}.
The dual graphs are $z$-knotted triangulations whose vertices are of degree $5$ or $6$.
Observe that each vertex of  degree $5$ is adjacent to every vertex of a certain $5$-gone,
but there is no other vertex adjacent to all vertices of this $5$-gone. 
This means that 
these triangulations cannot be presented as the connected sums of some collections of $z$-knotted bipyramids.
}\end{rem}

\begin{rem}{\rm
If $E$ is the edge set of a $z$-knotted embedded graph,
then the cyclic ordering of $E$ (associated to the zigzag) induces 
two linear transformations of $2^{E}$ (considered as a vector space over ${\mathbb F}_{2}$) \cite{LinsSilva}.
It will be interesting to describe relations between such transformations for $\Gamma, \Gamma'$ and 
$\Gamma \#_{g}\Gamma'$ (if this connected sum is $z$-knot\-ted).
}\end{rem}

\section{Proof of Theorem 1}

\subsection{General construction of zigzags in the connected sum}
From this moment, we will suppose that $\Gamma$ and $\Gamma'$ are $z$-knotted triangulations.
Let $F$ and $F'$ be faces in $\Gamma$ and $\Gamma'$, respectively.
If $x,y,z$ are the vertices of $F$, then the zigzag of $\Gamma$ is a sequence 
$$x,y,z,\dots,\delta(x),\delta(y),\delta(z),\dots,\gamma(x),\gamma(y),\gamma(z),\dots,$$
where $\delta$ and $\gamma$ are permutations on the set $\{x,y,z\}$.
Let us consider this zigzag as the union of the following three segments 
$$A=\{y, z, \dots, \delta(x),\delta(y)\},$$
$$B=\{\delta(y),\delta(z), \dots, \gamma(x),\gamma(y)\},$$
$$C=\{\gamma(y),\gamma(z) \dots, x,y\},$$
i.e. as the cyclic sequence $A,B,C$.  
It must be pointed out that  any two consecutive segments have the same vertex
(for example, the segments $A$ and $B$ are joined in the vertex $\delta(y)$).
Similarly, if $x',y',z'$ are the vertices of $F'$, then the zigzag of $\Gamma'$ is a sequence 
$$x',y',z',\dots,\delta'(x'),\delta'(y'),\delta'(z'),\dots,\gamma'(x'),\gamma'(y'),\gamma'(z'),\dots,$$
where $\delta'$ and $\gamma'$ are permutations on the set $\{x',y',z'\}$.
We consider this zigzag as the cyclic sequence $A',B',C'$, where
$$A'=\{y', z', \dots, \delta'(x'),\delta'(y')\},$$
$$B'=\{\delta'(y'),\delta'(z'), \dots, \gamma'(x'),\gamma'(y')\},$$
$$C'=\{\gamma'(y'),\gamma'(z') \dots, x',y'\}.$$
Let $g:\partial F\to \partial F'$ be a homeomorphism transferring vertices to vertices. 
Zigzags in the connected sum $\Gamma\#_{g}\Gamma'$ can be constructed as follows.
We start from a segment $X\in \{A,B,C\}$.
Let $a,b\in \{x,y,z\}$ be the last two vertices in $X$ (the order of the vertices is important).
There is a unique segment 
$$X'\in \{A',B',C',A'^{-1},B'^{-1},C'^{-1}\}$$
such that $g(a),g(b)\in \{x',y',z'\}$ are the first two vertices in $X'$ 
(as above, the order of the vertices is important).
Indeed, if $c'$ is the element of $\{x',y',z'\}$ distinct from $g(a),g(b)$ and
there are two segments satisfying this condition,
then the zigzag in $\Gamma'$ contains the sequence $g(a),g(b),c'$ twice or 
it contains this sequence together with the reversed sequence $c',g(b),g(a)$; each of these cases is impossible.

Let $a',b'\in \{x',y',z'\}$ be the last two vertices in $X'$.
There is a unique segment 
$$Y\in \{A,B,C,A^{-1},B^{-1},C^{-1}\}\setminus\{X, X^{-1}\}$$
such that $g^{-1}(a'),g^{-1}(b')$ are the first two vertices in $Y$.
If $Y=X$, then $X,X'$ is a zigzag in $\Gamma \#_{g} \Gamma'$.
In the case when $X$ and $Y$ are distinct, we apply the same arguments to $Y$ and, as above, choose a certain segment
$$Y'\in \{A',B',C',A'^{-1},B'^{-1},C'^{-1}\}\setminus \{X',X'^{-1}\}.$$  
Step by step, we construct a cyclic sequence which defines a zigzag in $\Gamma \#_{g} \Gamma'$.
The intersection of any two consecutive segments in this sequence  consists of precisely two vertices.
There are the following three possibilities for the associated zigzag:
$$X,X'\;\mbox{ or }\; X,X',Y,Y'\;\mbox{ or }\; X,X',Y,Y',Z,Z',$$
where $X,Y,Z$  and $X',Y',Z'$ are distinct elements of 
$$\{A,B,C,A^{-1},B^{-1},C^{-1}\}\;\mbox{ and }\;\{A',B',C',A'^{-1},B'^{-1},C'^{-1}\},$$
respectively.
In the third case, the connected sum $\Gamma \#_{g} \Gamma'$ is $z$-knotted.
Otherwise, $\Gamma \#_{g} \Gamma'$ contains three zigzags of type $X,X'$ 
or one zigzag of type $X,X'$ and one zigzag of type $Y,Y',Z,Z'$.

\subsection{Proof of the statement {\rm (1)}}
In this subsection, we suppose that $F$ and $F'$ both are $(1,1,2)$-faces, 
i.e. similar to faces in the $z$-knotted bipyramids. 
As in Example 2, 
we denote by $a$ and $a'$ the vertices of $F$ and $F'$ (respectively) which do not belong to the edges of second type.
Also, we write $1,2$ and $1',2'$ for the vertices on the edge of second type in $F$ and $F'$, respectively.
We will assume that the zigzags in $\Gamma$ and $\Gamma'$ go from $1$ and $1'$ to $2$ and $2'$, respectively.

{\it Case 1}
(${\rm od}+{\rm od}$).
If $F$ and $F'$ are of odd type,
then the zigzags in $\Gamma$ and $\Gamma'$ are sequences
$$a,1,2,\dots,1,a,2,\dots,1,2,a,\dots$$
and
$$a',1',2',\dots,1',a',2',\dots,1',2',a',\dots$$
(respectively) and we have  
$$A=\{1, 2, \dots, 1, a\},\;B=\{a, 2, \dots, 1, 2\},\; C=\{2, a, \dots, a, 1\},$$
$$A'=\{1', 2', \dots, 1', a'\},\;B'=\{a', 2', \dots, 1', 2'\},\; C'=\{2', a', \dots, a', 1'\}.$$
There are precisely $6$ distinct bijections between the vertex sets of  $F$ and $F'$. 
In the case when $a$ is identified with $a'$, the corresponding connected sums are $z$-knotted.
If $1,2$ are identified with $1',2'$ (respectively), then the unique zigzag in the connected sum is 
$$A, C'^{-1}, B, A', C^{-1}, B'.$$
If $1$ is identified with $2'$ and $2$ is identified with $1'$, then this is
$$A, C', B, B'^{-1}, C^{-1}, A'^{-1}.$$
In the remaining four cases, the connected sum contains precisely two zigzags and is not $z$-knotted (Tab.1).
\begin{center}
\begin{tabular}{|C|C|C|C|C|}
\hline
\multicolumn{1}{|c|}{Identification} & \multicolumn{2}{c|}{\multirow{2}{*}{Zigzags}} \\ 
\{g(a), g(1), g(2)\}           & \multicolumn{2}{l|}{}                         \\ \hline
\{1', 2', a'\}         & A, B'^{-1} & B, C', C, A' \\ \hline
\{1', a' ,2'\}         & B, B' & A, A'^{-1}, C, C'^{-1} \\ \hline
\{2', a', 1'\}         & B, A'^{-1} & A, B', C, C' \\ \hline
\{2', 1', a'\}         & A, A' & B, C'^{-1}, C, B'^{-1} \\ \hline
\end{tabular}
\captionof{table}{ }
\end{center}

\begin{exmp}{\rm
Suppose that $\Gamma$ is the $n$-bipyramid graph considered in Example 2, $n=2k+1$ and $k$ is odd.
Then
$$A=\{1, 2, b,\dots, 1, a\},\;
B=\{a, 2, \dots, b, 1, 2\},$$
$$C=\{2, a,\dots,1, b, 2, \dots, a, 1\}.$$
If $a, 1, 2$ are identified with $a', 1', 2'$ (respectively),  then the unique zigzag in the connected sum is $A, C'^{-1}, B, A', C^{-1}, B'$. 
The face containing $b, 1, 2$ appears in this zigzag as follows
$$\underbrace{1, 2, b,\dots}_{A}, C'^{-1}, \underbrace{\dots,b, 1, 2}_{B}, A',\underbrace{\dots, 2, b, 1, \dots}_{C^{-1}},B'$$
which means that it is a  $(2,2,2)$-face of first type.
The same holds for the case when $a, 1, 2$ are identified with $a', 2', 1'$, respectively. 
}\end{exmp}

{\it Case 2} (${\rm ev}+{\rm ev}$).
If $F$ and $F'$ are of even type, then the zigzags in $\Gamma$ and $\Gamma'$ are sequences
$$a,1,2,\dots,1,2,a,\dots,1,a,2,\dots$$
and
$$a',1',2',\dots,1',2',a',\dots,1',a',2',\dots$$
(respectively) and 
$$A=\{1, 2, \dots, 1, 2\},\;B=\{2, a, \dots, 1, a,\},\;C=\{a, 2, \dots, a, 1\},$$
$$A'=\{1', 2', \dots, 1', 2'\},\;B'=\{2', a', \dots, 1', a'\},\;C'=\{a', 2', \dots, a', 1'\}.$$
In contrast to the previous case, the connected sum is $z$-knotted if $a$ is not identified with $a'$
(Tab.2).
\begin{center}
\begin{tabular}{|C|C|C|C|C|}
\hline
\multicolumn{1}{|c|}{Identification} & {\multirow{2}{*}{Zigzag}} \\ 
\{g(a), g(1), g(2)\}           &       \\ \hline
\{1', 2', a'\}         & A, B', C, A', B^{-1}, C'^{-1} \\ \hline
\{1', a' ,2'\}         & A, C', C^{-1}, A'^{-1}, B, B'^{-1} \\ \hline
\{2', a', 1'\}         & A, B'^{-1}, C^{-1}, A', B, C' \\ \hline
\{2', 1', a'\}         & A, C'^{-1}, C, A'^{-1}, B^{-1}, B' \\ \hline
\end{tabular}
\captionof{table}{ }
\end{center}
In the case when $a$ is identified with $a'$, the connected sum contains precisely three zigzags and is not $z$-knotted.
If $1,2$ is identified with $1',2'$ (respectively), then these zigzags are 
$A, A'$ and $B, C'^{-1}$ and $C, B'^{-1}$.
If $1$ is identified with $2'$ and $2$ is identified with $1'$, then we get the zigzags
$A,A'^{-1}$ and $B,B'$ and $C,C'$.

\begin{exmp}{\rm
Suppose that $\Gamma$ is the $n$-bipyramid graph considered in Example 2, $n=2k+1$ and $k$ is even.
Then
$$A=\{1, 2, \dots, n-1, n, b, 1, 2\},\;B=\{2, a, \dots, n-1, b, n, 1, a\},$$
$$C=\{a, 2, \dots, b, n-1, n, a, 1\}.$$
If $a, 1, 2$ are identified with $1', 2', a'$ (respectively), then the unique zigzag in the connected sum is 
$A, B', C, A', B^{-1}, C'^{-1}$. 
The face containing $b, n-1, n$ appears in this zigzag as follows
$$\underbrace{\dots,n-1, n, b,\dots}_{A}, B',\underbrace{\dots,b, n-1, n, \dots}_{C},A', \underbrace{\dots,n, b, n-1,\dots}_{B^{-1}},C'^{-1},$$
i.e. this is a $(2,2,2)$-face of first type.
The same holds for the other three cases when $a$ is not identified with $a'$.
}\end{exmp}

{\it Case 3} (${\rm ev}+{\rm od}$).
If $F$ is of even type and $F'$ is of odd type, then 
$$A=\{1, 2, \dots, 1, 2\},\;B=\{2, a, \dots, 1, a,\},\;C=\{a, 2, \dots, a, 1\},$$
$$A'=\{1', 2', \dots, 1', a'\},\;B'=\{a', 2', \dots, 1', 2'\},\;C'=\{2', a', \dots, a', 1'\}.$$
The connected sum is $z$-knotted if $a$ is identified with $a'$ (Tab.3).
\begin{center}
\begin{tabular}{|C|C|C|C|C|}
\hline
\multicolumn{1}{|c|}{Identification} & {\multirow{2}{*}{Zigzag}} \\ 
\{g(a), g(1), g(2)\}           &       \\ \hline
\{a', 1', 2'\}         & A, A', C^{-1}, C', B^{-1}, B' \\ \hline
\{a', 2' ,1'\}         & A, B'^{-1}, C^{-1}, C'^{-1}, B^{-1}, A'^{-1} \\ \hline
\end{tabular}
\captionof{table}{ }
\end{center}
For each of the remaining four cases, the connected sum contains precisely two zigzags and is not $z$-knotted
(Tab.4).
\begin{center}
\begin{tabular}{|C|C|C|C|C|}
\hline
\multicolumn{1}{|c|}{Identification} & \multicolumn{2}{c|}{\multirow{2}{*}{Zigzags}} \\ 
\{g(a), g(1), g(2)\}           & \multicolumn{2}{l|}{}                         \\ \hline
\{1', 2', a'\}         & C, A' & A, C', B, B'^{-1} \\ \hline
\{1', a' ,2'\}         & B, A'^{-1} & A, B', C, C'^{-1} \\ \hline
\{2', a', 1'\}         & B, B' & A, A'^{-1}, C, C' \\ \hline
\{2', 1', a'\}         & C, B'^{-1} & A, C'^{-1}, B, A' \\ \hline
\end{tabular}
\captionof{table}{ }
\end{center}

\begin{rem}{\rm
Suppose that $\Gamma$ and $\Gamma'$ are the $n$-bipyramid and $n'$-bipyramid graphs (respectively),
where $n=2k+1$, $n'=2k'+1$, $k$ is odd and $k'$ is even. 
Then all faces in any $z$-knotted connected sum of $\Gamma$ and $\Gamma'$ are $(1,1,2)$-faces.
}\end{rem}

\subsection{Proof of the statement {\rm (2)}}
In this subsection, we suppose that $F$ is a $(2,2,2)$-face of first type.
We denote by $a,b,c$ the vertices of $F$ and assume that the orientation on $F$ defined by the zigzag coincides with 
the order of vertices in the sequence $a,b,c$.
Then the zigzag is a sequence 
$$a,b,c,\dots, c,a,b,\dots,b,c,a,\dots$$
and we have
$$A=\{b, c, \dots, c, a\},\;B=\{a, b, \dots, b, c\},\;C=\{c, a, \dots, a, b\}.$$
We show that for every homeomorphism $g:\partial F\to \partial F'$ the connected sum $\Gamma \#_{g} \Gamma'$ is $z$-knotted.

Consider the case when $F'$ is a $(1,1,2)$-face.
As in the previous subsection, we assume that $a'$ is the vertex which do not belong to the edges of second type,
$1'$ and $2'$ are the vertices on the edge of second type and the zigzag goes from $1'$ to $2'$.
If $F'$ is of odd type, then
$$A'=\{1', 2', \dots, 1', a'\},\;B'=\{a', 2', \dots, 1', 2'\},\; C'=\{2', a', \dots, a', 1'\}$$
and the zigzags in the corresponding connected sums are presented in Tab.5.

\begin{center}
\begin{tabular}{|C|C|C|C|C|}
\hline
\multicolumn{1}{|c|}{Identification} & {\multirow{2}{*}{Zigzag}} \\ 
\{g(a), g(b), g(c)\}           &       \\ \hline
\{a', 1', 2'\}         & A,C', B, A', C^{-1}, B' \\ \hline
\{a', 2', 1'\}         & A,C'^{-1},B,B'^{-1},C^{-1},A'^{-1} \\ \hline
\{1', 2', a'\}         & A,A'^{-1},C^{-1},C'^{-1},B^{-1},B'^{-1} \\ \hline
\{1', a', 2'\}         & A,B'^{-1},B^{-1},A'^{-1},C,C'^{-1} \\ \hline
\{2', a', 1'\}         & A,A',B^{-1},B',C,C' \\ \hline
\{2', 1', a'\}         & A,B',C^{-1},C',B^{-1},A' \\ \hline
\end{tabular}
\captionof{table}{ }
\end{center}
In the case when $F'$ is of even type, we have           
$$A'=\{1', 2', \dots, 1', 2'\},\;B'=\{2', a', \dots, 1', a'\},\;C'=\{a', 2', \dots, a', 1'\}$$
and the zigzags in the connected sums can be found in Tab.6.
\begin{center}
\begin{tabular}{|C|C|C|C|C|}
\hline
\multicolumn{1}{|c|}{Identification} & {\multirow{2}{*}{Zigzag}} \\ 
\{g(a), g(b), g(c)\}           &       \\ \hline
\{a', 1', 2'\}         & A,B', C^{-1},C',B,A' \\ \hline
\{a', 2', 1'\}         & A,C'^{-1},C^{-1},B'^{-1},B,A'^{-1} \\ \hline
\{1', 2', a'\}         & A,B'^{-1},B^{-1},A'^{-1},C^{-1},C'^{-1} \\ \hline
\{1', a', 2'\}         & A,A'^{-1},C,C'^{-1},B^{-1},B'^{-1} \\ \hline
\{2', a', 1'\}         & A,A',C,B',B^{-1},C' \\ \hline
\{2', 1', a'\}         & A,C',B^{-1},A',C^{-1},B' \\ \hline
\end{tabular}
\captionof{table}{ }
\end{center}

Suppose that $F'$ is a $(2,2,2)$-face whose vertices are $a',b',c'$ and 
the orientation coincides with the order of vertices in the sequence $a',b',c'$.
If $F'$ is of first type, then 
$$A'=\{b', c', \dots, c', a'\},\;B'=\{a', b', \dots, b', c'\},\;C'=\{c', a', \dots, a', b'\}$$
and the zigzags in the connected sums are described in Tab.7.
\begin{center}
\begin{tabular}{|C|C|C|C|C|}
\hline
\multicolumn{1}{|c|}{Identification} & {\multirow{2}{*}{Zigzag}} \\ 
\{g(a), g(b), g(c)\}           &       \\ \hline
\{a', b', c'\}         & A, C', B, A', C, B' \\ \hline
\{b', c', a'\}         & A, B', B, C', C, A' \\ \hline
\{c', a', b'\}         & A, A', B, B', C, C' \\ \hline
\{c', b', a'\}         & A, A'^{-1}, B, C'^{-1}, C, B'^{-1} \\ \hline
\{b', a', c'\}         & A, B'^{-1}, B, A'^{-1}, C, C'^{-1} \\ \hline
\{a', c', b'\}         & A, C'^{-1}, B, B'^{-1}, C, A'^{-1} \\ \hline
\end{tabular}
\captionof{table}{ }
\end{center}
If $F'$ is of second type, then the zigzag in $\Gamma'$ is a sequence
$$a',b',c',\dots,b',c',a',\dots,c',a',b',\dots$$
and 
$$A'=\{b', c', ..., b', c'\},\;B'=\{c', a', ..., c', a'\},\;C'=\{a', b', ..., a', b'\}.$$
The zigzags in the connected sums are presented in  Tab.8.
\begin{center}
\begin{tabular}{|C|C|C|C|C|}
\hline
\multicolumn{1}{|c|}{Identification} & {\multirow{2}{*}{Zigzag}} \\ 
\{g(a), g(b), g(c)\}           &       \\ \hline
\{a', b', c'\}         & A, B', C, C', B, A' \\ \hline
\{b', c', a'\}         & A, C', C, A', B, B' \\ \hline
\{c', a', b'\}         & A, A', C, B', B, C' \\ \hline
\{c', b', a'\}         & A, B'^{-1}, C, A'^{-1}, B, C'^{-1} \\ \hline
\{b', a', c'\}         & A, A'^{-1}, C, C'^{-1}, B, B'^{-1} \\ \hline
\{a', c', b'\}         & A, C'^{-1}, C, B'^{-1}, B, A'^{-1} \\ \hline
\end{tabular}
\captionof{table}{ }
\end{center}

\subsection{Proof of the statement {\rm (3)}}
In this subsection, we suppose that $F$ is a $(2,2,2)$-face of second type.
As in the previous subsection, we denote by $a,b,c$ the vertices of $F$ and 
assume that the orientation on $F$ defined by the zigzag coincides with  the order of vertices in the sequence $a,b,c$.
Then
$$A=\{b, c, ..., b, c\},\;B=\{c, a, ..., c, a\},\;C=\{a, b, ..., a, b\}.$$
If $F'$ is a $(2,2,2)$-face of first type, then  the connected sum $\Gamma \#_{g} \Gamma'$ is $z$-knotted for 
any homeomorphism $g:\partial F \to \partial F'$ (see the previous subsection).
If $F'$ is a $(2,2,2)$-face of second type, then 
$$A'=\{b', c', ..., b', c'\},\;B'=\{c', a', ..., c', a'\},\;C'=\{a', b', ..., a', b'\}.$$
In this case, for every of the segments $A,B,C,A',B',C'$ the first two vertices coincides with the last two vertices.
Therefore, for any homeomorphism $g:\partial F \to \partial F'$  all zigzags in the connected sum $\Gamma \#_{g} \Gamma'$ are of type $X,X'$,
where 
$$X\in \{A,B,C,A^{-1},B^{-1},C^{-1}\}\;\mbox{ and }\; X'\in \{A',B',C',A'^{-1},B'^{-1},C'^{-1}\}.$$
So, the connected sum is not $z$-knotted. 

Consider the case when $F'$ is a $(1,1,2)$-face whose vertices are $a',1',2'$ and 
$a'$ is the vertex which does not belong to the edge of second type.
As above, we suppose that the zigzag goes from $1'$ to $2'$.
If $F'$ is of odd type, then 
$$A'=\{1', 2', ..., 1', a'\},\;B'=\{a', 2', ..., 1', 2'\},\;C=\{2', a', ..., a', 1'\}$$
and for every homeomorphism $g:\partial F\to \partial F'$ the connected sum $\Gamma \#_{g} \Gamma'$ is $z$-knotted (Tab. 9).
\begin{center}
\begin{tabular}{|C|C|C|C|C|}
\hline
\multicolumn{1}{|c|}{Identification} & {\multirow{2}{*}{Zigzag}} \\ 
\{g(a), g(b), g(c)\}           &       \\ \hline
\{a', 1', 2'\}         & A,A',C^{-1},C'^{-1},B^{-1},B' \\ \hline
\{a', 2', 1'\}         & A,B'^{-1},C^{-1},C',B^{-1},A'^{-1} \\ \hline
\{1', 2', a'\}         & A, C',B,A'^{-1},C^{-1},B'^{-1} \\ \hline
\{1', a', 2'\}         & A,B',B^{-1}, A',C,C'^{-1} \\ \hline
\{2', a', 1'\}         & A,A'^{-1},B^{-1},B'^{-1},C,C' \\ \hline
\{2', 1', a'\}         & A,C'^{-1},B,B',C^{-1},A' \\ \hline
\end{tabular}
\captionof{table}{ }
\end{center}
If $F'$ is of even type, then 
$$A'=\{1', 2', \dots, 1', 2'\},\;B'=\{2', a', \dots, 1', a',\},\;C'=\{a', 2', \dots, a', 1'\}.$$
and for every homeomorphism $g:\partial F\to \partial F'$ the connected sum $\Gamma \#_{g} \Gamma'$ is not $z$-knotted (Tab.10).
\begin{center}
\begin{tabular}{|C|C|C|C|C|}
\hline
\multicolumn{1}{|c|}{Identification} & \multicolumn{2}{c|}{\multirow{2}{*}{Zigzags}} \\ 
\{g(a), g(b), g(c)\}           & \multicolumn{2}{l|}{}                         \\ \hline
\{a', 1', 2'\}         & A, A' & B, B', C^{-1}, C'^{-1} \\ \hline
\{a', 2', 1'\}         & A, A'^{-1} & B, C'^{-1}, C^{-1}, B' \\ \hline
\{1', 2', a'\}         & C, A' & A, B', B^{-1}, C'^{-1} \\ \hline
\{1', a', 2'\}         & B, A'^{-1} & A, C', C^{-1}, B'^{-1} \\ \hline
\{2', a', 1'\}         & B, A' & A, B'^{-1}, C^{-1}, C' \\ \hline
\{2', 1', a'\}         & C, A'^{-1} & A, C'^{-1}, B^{-1}, B' \\ \hline
\end{tabular}
\captionof{table}{ }
\end{center}

\section{$(2,2,2)$-faces of second type}
Consider the bipyramid graph $BP_{n}$, where $n=2k$ and $k$ is odd.
As in Example 2, we denote by $1,\dots,n$ the vertices on the $n$-gone
and write $a,b$ for the remaining two vertices.
There precisely two zigzags in $BP_{n}$. These  are the sequences 
$$a, 1, 2, b, 3, 4, \dots, a, n-1, n, b, 1, 2, a, 3, 4, \dots, b, n-1, n$$
and
$$a, 2, 3, b, \dots, b, n-2, n-1, a, n, 1, b, 2, 3, a, \dots, a, n-2, n-1, b, n, 1.$$
Let $F$ be the face containing $a,1,2$.
Following Subsection 6.1, we present zigzags as the unions of segments
which are parts of zigzags between two edges of $F$, i.e.
each of these segments contains precisely two edges of $F$.
The first zigzag is the sum of the following two segments
$$A=\{1, 2, \dots, 1, 2\}\;\mbox{ and }\;B=\{2, a, 3, \dots, a, 1\}.$$
The second zigzag passes ones through the edges $a1$ and $a2$, and it does not contain the edge $12$.
So, it will be identified with the segment 
$$C=\{a, 2, 3, \dots, 2, 3, a, \dots, 1, a\}.$$

Now, we take the bipyramid graph $BP_{n'}$, where $n'=2k'$ and $k'$ is odd.
Let $1',\dots,n'$ be the vertices on the $n'$-gone, and let $a',b'$ be the remaining two vertices.
Denote by $F'$ the face containing $a',1',2'$.
As above, one of the zigzags is the union of two segments $A',B'$ and the other zigzag is identified with the segment $C'$.

We identify the vertices $a, 1, 2$ with the vertices $2', a', 1'$ (respectively).
The corresponding connected sum of $BP_{n}$ and $BP_{n'}$ contains the unique zigzag
$$A, C'^{-1}, C^{-1}, A', B, B',$$
i.e. it is $z$-knotted. 
The face containing $a, 2, 3$ appears in this zigzag as follows
$$A, C'^{-1}, \underbrace{\dots, a, 3, 2, \dots, 3, 2, a}_{C^{-1}}, A', \underbrace{2, a, 3, \dots}_{B}, B'.$$
This is a $(2,2,2)$-face of second type.

\end{document}